\documentclass[12pt, reqno]{amsart}

\usepackage{amsmath, amsfonts, amssymb}

\usepackage[latin1]{inputenc}
\usepackage{latexsym}

\usepackage{graphics}
\usepackage{color}

\newtheorem{theorem}{{\sc Theorem}}[section]

\newtheorem{cor}[theorem]{{\sc Corollary}}

\newtheorem{lemma}[theorem]{{\sc Lemma}}

\newtheorem{assumption}[theorem]{{\sc Assumption}}
\newtheorem{definition}[theorem]{{\sc Definition}}
\theoremstyle{remark}
\newtheorem{remark}[theorem]{{\sc Remark}}

\newtheorem{example}[theorem]{{\sc Example}}
\newcommand{\E}{\mathbb E}
\newcommand{\V}{\mathbb V}

\newcommand{\R}{\mathbb R}
\newcommand{\N}{\mathbb N}
\newcommand{\Z}{\mathbb Z}

\newcommand{\re}{{\mathbb R}}    
\newcommand{\cn}{{\mathbb C}}
\newcommand{\qn}{{\mathbb H}}    

\begin{document}
\title[Moment estimates of Rosenthal type via cumulants]{Moment estimates of Rosenthal type via cumulants}
\bigskip
\author[Peter Eichelsbacher, Lukas Knichel]{}

\maketitle
\thispagestyle{empty}
\vspace{0.2cm}

\centerline{\sc Peter Eichelsbacher\footnote{ Ruhr-Universit\"at Bochum, Fakult\"at f\"ur Mathematik, 
IB 2/115, D-44780 Bochum, Germany,{\tt peter.eichelsbacher@rub.de  }}, 
Lukas Knichel\footnote{Ruhr-Universit\"at Bochum, Fakult\"at f\"ur Mathematik, 
IB 2/95, D-44780 Bochum, Germany, {\tt lukas.knichel@ruhr-uni-bochum.de  } \\The second author has been supported by Deutsche Forschungsgemeinschaft via RTG 2131. %
}}
\vspace{0.5cm}

\vspace{2 cm}

\pagenumbering{roman}
\maketitle

\pagenumbering{arabic}
\pagestyle{headings}

{ {\bf Abstract:} } The purpose of the present paper is to establish moment estimates of Rosenthal type
for a rather general class of random variables satisfying certain bounds on the cumulants. We consider
sequences of random variables which satisfy a central limit theorem and estimate the speed of convergence
of the corresponding moments to the moments of a standard normally distributed variable.
The examples of random objects we discuss include those where a dependency graphs or a weighted dependency graph encodes the dependency structure. We give applications to subgraph-counting statistics in Erd\H{o}s-R\'enyi random graphs
of type $G(n,p)$ and $G(n,m)$, crossings in uniform random pair partitions and spins in the $d$-dimensional Ising model.
Moreover, we prove moment estimates for certain statistics appearing in random matrix theory,
namely characteristic polynomials of random unitary matrices as well as the determinants of certain random matrix ensembles.
We add estimates for the $p(n)$-dimensional volume of the simplex with $p(n)+1$ points in $\R^n$ distributed according
to special distributions, since it is strongly connected to Gram matrix ensembles.
\bigskip
\bigskip

\section{Introduction and main theorem}

\subsection{Cumulants}
Since the late seventies estimations of cumulants have not only been studied
to show convergence in law, but also to investigate
a more precise asymptotic analysis of the distribution via
the rate of convergence and large deviation principles,
see e.g. \cite{SaulisStratulyavichus:1989} and references therein.
In \cite{ERS:2015} and \cite{DE2013} it has been shown how to relate these bounds to prove a
moderate deviation principle for quite a large class of random models.
This paper provides a general approach to show moment estimates via cumulants.

Let $X$ be a real-valued random variable with existing absolute moments. Then
$$
\left. \Gamma_j := \Gamma_j(X) :=(-i)^j \frac{d^j}{dt^j} \log \E\bigl[e^{i t X}\bigr] \right|_{t=0}
$$
exists for all $j\in\mathbb N$ and the term is called the {\it $j$th cumulant}
(also called semi-invariant) of $X$. Here and in the following, $\E$ denotes the expectation of the corresponding random variable and $\V$ its variance.
The method of moments results in a method of cumulants, saying that
if the distribution of $X$ is determined by its moments and $(X_n)_n$ are random variables with finite moments
such that $\Gamma_j(X_n) \to \Gamma_j(X)$ as $n \to \infty$ for every $j \geq 1$, then $(X_n)_n$ converges in distribution to $X$.
Hence if the first cumulant of $X_n$ converges to zero, the second cumulant to one,
and all cumulants of $X_n$ of order bigger than $2$ vanish, then the sequence $(X_n)_n$  satisfies a Central Limit Theorem (CLT). 
Knowing, in addition, exact bounds of the cumulants one is able to describe
the asymptotic behaviour more precisely.  Let $Z_n$ be a real-valued random
variable with mean $\E Z_n=0$ and variance $\V Z_n=1$, and
\begin{equation} \label{cum1}
|\Gamma_j(Z_n)| \leq \frac{(j!)^{1+\gamma}}{\Delta_n^{j-2}}
\end{equation}
for all $j=3,4, \ldots$, $n \geq 1$, for fixed $\gamma\geq 0$ and $\Delta>0$. 
Denoting the standard normal distribution function by
$$\Phi(x):=\frac{1}{\sqrt{2\pi}} \int_{-\infty}^x e^{-\frac{y^2}{2}}dy\,,$$
one obtains the following bound for the Kolmogorov distance
$$
\sup_{x\in\mathbb R}\bigl|P(Z_n\leq x)-\Phi(x)\bigr|
\leq c_{\gamma} \, \Delta_n^{- \frac{1}{1+2\gamma}},
$$
where $c_{\gamma}$ is a constant depending only on $\gamma$, see \cite[Lemma 2.1]{SaulisStratulyavichus:1989}.
By this result, the distribution function $F_n$ of $Z_n$ converges uniformly to $\Phi$ as $n \to \infty$.
Hence, when $x=O(1)$ we have 
\begin{equation} \label{mainratio}
\lim_{n \to \infty} \frac{1-F_n(x)}{1 -\Phi(x)} = 1.
\end{equation} 
One is interested to have -- under additional conditions -- 
such a relation in the case when $x$ depends on $n$ and tends to $\infty$ as $n \to \infty$. In particular,
one is interested in conditions for which the relation \eqref{mainratio} holds in the interval 
$0 \leq x \leq f(n)$, where $f(n)$ is a non-decreasing function such that $f(n) \to \infty$. 
If the relation hold in such an interval, we call the interval {\it a zone of normal convergence}. 
 
For i.i.d. partial sums, the classical result due to Cram\'er is that if $\E e^{t |X_1|^{1/2}} < \infty$
for some $t >0$, \eqref{mainratio} holds with $f(n) = o(n^{1/6})$.
In \cite[Chapter 2]{SaulisStratulyavichus:1989}, relations of large deviations of the type \eqref{mainratio} are proved under the
condition \eqref{cum1} on cumulants with a zone of normal convergence of size proportional to $\Delta^{\frac{1}{1 + 2 \gamma}}$, see
Lemma 2.3 in \cite{SaulisStratulyavichus:1989}. 

\subsection{Rosenthal-type inequalities}
The aim of this paper is to show that under the same type of condition on cumulants of random variables $Z_n$, moment inequalities of Rosenthal type can be deduced. For independent random variables, the Rosenthal
inequalities relate moments of order higher than 2 of partial sums of random variables to the variance of partial sums.
In \cite{Rosenthal:1970} it was proved that for $(X_k)_k$ being an independent and centered sequence of real valued random variables with finite moments of order $p$, $p \geq 2$, one obtains for every positive integer $n$ that
$$
\E \bigl( \big| \sum_{j=1}^n X_j \big|^p \bigr) \ll \sum_{j=1}^n \E \bigl( |X_j|^p \bigr) + \biggl( \sum_{j=1}^n E(X_j^2) \biggr)^{\frac p2}.
$$
Here $a_n \ll b_n$ means that there exists a numerical constant $C_p$, depending only on $p$ (and not on the underlying random variables nor on $n$),
such that $a_n \leq C_p b_n$ for all positive integers $n$. A first Rosenthal-type inequality for weakly dependent random variables was derived in 
\cite{DoukhanLouhichi:1999}. In \cite{DoukhanNeumann:2007} cumulant estimates are employed for deriving inequalities of Rosenthal type for weakly dependent random variables.
Our abstract result, Theorem \ref{thmcumulants}, is motivated by this work.  
We will prove moment estimates for a couple of statistics applying Theorem \ref{thmcumulants}. 

The following is the main result of the paper:

\begin{theorem}\label{thmcumulants}
For any $n \in {\Bbb N}$, let $Z_n$ be a centered random variable with variance one and existing
absolute moments, which satisfies
\begin{equation}\label{eqcumulants}
\bigl| \Gamma_j(Z_n) \bigr| \leq \frac{C_{j,\gamma}}{\Delta_n^{j-2}}
\quad\text{for all } j=3,4, \dots
\end{equation}
for a constant $C_{j, \gamma}$ depending on $j$ and a fixed $\gamma\geq 0$ and $\Delta_n>0$. Then for any $k = 3, 4,\ldots$ we obtain
$$
\big| \E (Z_n^k) - \E (N^k) \big|  \leq \sum_{1 \leq j \leq \lceil \frac k2 -1 \rceil } A_{j,k} \frac{1}{\Delta_n^{k -2j}},
$$
where $\lceil \cdot \rceil$ denotes the ceiling function, meaning that $\lceil \frac k2 -1 \rceil = \frac k2 -1$ if $k$ is even and $\lceil \frac k2 -1 \rceil = \frac k2 - \frac 12 $ when $k$ is odd, and
$$
A_{j,k} := \frac{1}{j!} \sum_{k_1 + \cdots + k_j =k, k_i \geq 2} C_{k_1, \gamma} \cdots C_{k_j, \gamma} \frac{k!}{k_1! \cdots k_j!},
$$
and $N$ denotes a standard normally distributed random variable. For an even $k=2l$, assuming that $\Delta_n \to \infty$ we obtain
$$
\big| \E (Z_n^{k}) - \E (N^{k}) \big|  \leq C_1(k,\gamma) 
\frac{1}{\Delta_n^2}
$$
with a constant $C_1(k, \gamma)$ only depending on $k$ and $\gamma$.
For an odd $k=2l+1$, assuming that $\Delta_n \to \infty$ we obtain
$$
\big| \E (Z_n^{k}) \big|  \leq C_2(k,\gamma) 
\frac{1}{\Delta_n}
$$
with a constant $C_2(k, \gamma)$ only depending on $k$ and $\gamma$.
\end{theorem}

\begin{cor} \label{classical}
For any $n \in {\Bbb N}$, let $Z_n$ be a centered random variable with variance one and existing
absolute moments, which satisfies
\begin{equation} \label{eqcumulants2}
\bigl| \Gamma_j(Z_n) \bigr| \leq \frac{(j!)^{1+\gamma} \widetilde{C}_j}{\Delta_n^{j-2}}
\quad\text{for all } j=3,4, \dots
\end{equation}
Then for any $k = 3, 4,\ldots$ we obtain
$$
\big| \E (Z_n^k) - \E (N^k) \big|  \leq (k!)^{1 + \gamma} \sum_{1 \leq j \leq \lceil \frac k2 -1 \rceil} \tilde{A}_{j,k} \frac{1}{\Delta_n^{k -2j}}
$$
with
$$
\tilde{A}_{j,k} := \frac{1}{j!} \sum_{k_1 + \cdots + k_j =k, k_i \geq 2} \widetilde{C}_{k_1} \cdots \widetilde{C}_{k_j} \frac{k!}{k_1! \cdots k_j!}.
$$
\end{cor}

\begin{remark}
In our result, the rate of convergence of moments only depends on $\Delta_n$ but {\it not} on the value $\gamma$. The value
$\gamma$ only influences the size of the constants $C_i(k,\gamma)$. This is remarkable, since under condition \eqref{eqcumulants2}
the zone of normal convergence is of size $\Delta_n^{\frac{1}{1 + 2 \gamma}}$, heavily depending on $\gamma$.
\end{remark}

\begin{proof}
By our assumptions we have $\Gamma_1(Z_n)=0$ and $\Gamma_2(Z_n) =1$. We now apply a formula due to Leonov and Shiryaev \cite{LS} to express moments of order $k$ through the cumulants $\Gamma_1(Z_n), \ldots, \Gamma_k(Z_n)$:
\begin{equation} \label{LS}
\E (Z_n^k) = \sum_{j=1}^{[k/2]} \frac{1}{j!} \sum_{k_1 + \cdots + k_j =k} \frac{k!}{k_1! \cdots k_j!} \Gamma_{k_1}(Z_n) \cdots \Gamma_{k_j}(Z_n),
\end{equation}
see for example \cite[formula (1.53) on page 11]{SaulisStratulyavichus:1989}. Note that $\Gamma_1(Z_n) =0$ implies that the inner sums in \eqref{LS} can be reduced to indices such that $k_i \geq 2$ for all $i$. Let us assume that $k$ is an even number.
Now the summand with $j = \frac k2$ on the right hand side of \eqref{LS} is equal to
$$
\frac{k!}{2^{\frac k2} \bigl( \frac k2 \bigr)!} \bigl( \Gamma_2(Z_n) \bigr)^{\frac k2} = \E N^k.
$$
Now we apply \eqref{eqcumulants} and obtain
$$
\big| \Gamma_{k_1}(Z_n) \cdots \Gamma_{k_j}(Z_n) \big| \leq  
C_{k_1, \gamma} \cdots C_{k_j, \gamma} \frac{1}{\Delta_n^{k - 2j}}.
$$
With the definition of $A_{j,k}$ we obtain the result
$$
\big | \E (Z_n^k) - \E (N^k) \big| \leq  \sum_{1 \leq  j \leq \frac k2 -1} A_{j,k} \frac{1}{{\Delta_n^{k - 2j}}}.
$$
When $k$ is odd, then $\E N^k = 0$ and we have to sum all the way up to $[k/2] = k/2 - 1/2$. \\
If $k$ is even, the leading term in the bound is the summand with $j= \frac k2 -1$ yielding $\frac{1}{\Delta_n^2}$.
If $k$ is odd, the leading term in the bound is the summand with $j= \frac{k-1}{2}$ yielding $\frac{1}{\Delta_n}$.
\end{proof}

\begin{proof}[Proof of Corollary \ref{classical}]
With \eqref{eqcumulants2} we apply H\"older's inequality to the Gamma function to see that $ \bigl( k_j! \bigr)^{\frac{k}{k_j}} \leq k!$.
Hence $k_1! \cdots k_j! \leq k!^{\frac{k_1+ \cdots + k_j}{k}} = k!$. Summarizing we obtain
$$
\big| \Gamma_{k_1}(Z_n) \cdots \Gamma_{k_j}(Z_n) \big| \leq k!^{1 + \gamma}   \tilde{C}_{k_1} \cdots \tilde{C}_{k_j} 
\frac{1}{{\Delta_n^{k - 2j}}}.
$$
With \eqref{LS} the proof is the same as for Theorem \ref{thmcumulants}.
\end{proof}

\noindent
In \cite[Theorem 4]{Bahr:1965}, a first result on the convergence of moments for a partial sum of independent random variables was obtained. The results
were improved in \cite{Hall:1982} and \cite{Hall:1983}.
Results from \cite[p. 208]{BR:book} can be used to derive a rate of convergence in the classical central limit theorem for moments: let $(X_i)_i$ be an i.i.d. sequence
of random variables with zero mean and unit variance, and let $Z_n = \frac{1}{\sqrt{n}} \sum_{i=1}^n X_i$. If $0 < p < 4$ and $\E(X_1^4) < \infty$, and $X_1$ satisfies Cram\'er's continuity condition $\limsup_{t \to \infty} | \E (e^{itX_1}) | < \infty$, then Theorem 20.1 in \cite{BR:book} implies
$$
\E |Z_n|^p = \E |N|^p + c_p \frac 1n + o(n^{-1})
$$
as $n \to \infty$, where the constant $c_p$ depends only on $p$ and the first four moments of $X_1$.
\medskip

\noindent
Our Theorem \ref{thmcumulants} opens up the possibility to prove moment estimates for a wide range
of dependent random variables. Before we proceed, we start with a warm up: we consider
a partial sum of {\it independent, non-identically distributed} random variables.

\begin{theorem} \label{thmnotidentical}
Let $(X_i)_{i \geq 1}$ be a sequence of independent real-valued random variables with expectation zero and variances
$\sigma_i^2>0$, $i \geq 1$, and let us assume that there exit $\gamma\geq 0$ and $K>0$ such that for all $i \geq 1$
\begin{equation}\label{momentenbedingungen}
\bigl| \E X_i^j\bigr| \leq (j!)^{1+\gamma} K^{j-2} \sigma_i^2
\quad\text{for all } j=3,4, \dots\, .
\end{equation}
Let $Z_n:=\frac{1}{\sqrt{\sum_{i=1}^n \sigma_i^2}}\sum_{i=1}^n X_i$. Then we obtain for all $k \geq 2$
$$
\big| \E (Z_n^{2k}) - \E (N^{2k}) \big|  \leq C_1(k) 
\frac{{\displaystyle{4 \max\bigl\{K^2; \max_{1\leq i\leq n}\{\sigma_i^2\}\bigr\}}}}{\sum_{i=1}^n \sigma_i^2},
$$
and 
$$
\big| \E (Z_n^{2k+1}) \big|  \leq C_2(k) 
\frac{{\displaystyle{2 \max\bigl\{K; \max_{1\leq i\leq n}\{\sigma_i\}\bigr\}}}}{\bigl( \sum_{i=1}^n \sigma_i^2 \bigr)^{1/2}}.
$$
\end{theorem}

\noindent
Remark that condition \eqref{momentenbedingungen} is a generalization of the classical Bernstein condition ($\gamma =0$).
\begin{proof}
Using a relation between moments and cumulants, condition \eqref{momentenbedingungen} implies that the $j$-th cumulant of $X_i$ 
can be bounded by $(j!)^{1 + \gamma} (2 \max \{K, \sigma_i \})^{j-2} \sigma_i^2$. Hence it follows from the independence of the random variables $X_i$, $i \geq 1$,
that the $j$-th cumulant of $Z_n$ has the bound
\begin{equation}\label{cumulantenbedinungen}
|\Gamma_j(Z_n)| \leq (j!)^{1+\gamma} \left(
\frac{\displaystyle 2 \max\bigl\{K; \max_{1\leq i\leq n}\{\sigma_i\}\bigr\}}{\sqrt{\sum_{i=1}^n \sigma_i^2}}\right)^{j-2} 
\, ,
\end{equation}
for details see for example \cite[Theorem 3.1]{SaulisStratulyavichus:1989}.
Thus for $Z_n$, the condition of Theorem \ref{thmcumulants} holds with
$$
\Delta_n=\frac{\sqrt{\sum_{i=1}^n \sigma_i^2}}{\displaystyle{2 \max\bigl\{K; \max_{1\leq i\leq n}\{\sigma_i\}\bigr\}}}\,.
$$
The result follows from Theorem \ref{thmcumulants}.
\end{proof}

\begin{remark}
If Cram\'er's condition holds, that is there
exists $\lambda>0$ such that $\E e^{\lambda |X_i|} < \infty$
holds for all $i\in\mathbb N$, then $X_i$ satisfies Bernstein's condition, which is the bound \eqref{momentenbedingungen} 
with $\gamma=0$, see for example \cite[Remark 3.6.1]{Yurinsky:1995}.
This implies \eqref{cumulantenbedinungen} and we can apply Theorem \ref{thmcumulants}
as above. Therefore Theorem \ref{thmcumulants} requires less restrictions on the random sequence
than Cram\'er's condition.
\end{remark}

The paper is organized as follows. Section 2 is devoted to applications for models where a {\it dependency graph} encodes the dependency structure in a family of random variables. Examples include counting statistics of subgraphs in Erd\H{o}s-R\'enyi random graphs $G(n,p_n)$.  In Section 3 models are considered, where edges of a corresponding dependency graph have
a weight called {\it weighted dependency graph}. We analyse the number of crossings in a random pair partition,
subgraph count statistics in the Erd\H{o}s-R\'enyi graph $G(n,m_n)$, as well as the mean number of spins in the $d$-dimensional
Ising model. Section 4 is devoted to $U$-statistics, whereas in Section 5 and 6, our
Theorem \ref{thmcumulants}  will be applied in random matrix theory and in geometric probability. In Section 6 we 
derive cumulant bounds for the logarithm of the determinant of a large class of random matrix ensembles. Our examples present
the possible variety of rates of convergences like in the central limit theorems. The difference $ \big| \E(Z_n^k) - \E (N^k) \big|$
converges to zero with a speed like $1/n$, $1/ (n^2)$, $1/ (n^3)$, $1/(2n+1)^d$, with $d\geq 1$ fixed, $1/ (\log n)$ and $1/( n \, p(n))$ for certain $p(n)$.

\section{Uniform control on cumulants and Dependency graphs} 

Let us start with the definition of a dependency graph due to \cite{Janson:1988}:

\begin{definition} \label{defDG}
Let ${\{X_{\alpha}\}}_{\alpha\in \mathcal I}$ be a family of random variables defined
on a common probability space. A {\it dependency graph} for ${\{X_{\alpha}\}}_{\alpha\in \mathcal I}$
is any graph $L$ with vertex set $\mathcal I$ which satisfies the following condition:
For any two disjoint subsets of vertices $V_1$ and $V_2$ such that there is no edge from any vertex in
$V_1$ to any vertex in  $V_2$, the corresponding collections of random
variables $\{X_{\alpha}\}_{\alpha\in V_1}$ and $\{X_{\alpha}\}_{\alpha\in V_2}$ are independent, see \cite{Janson:1988}.
\end{definition}

Let the {\it maximal degree} of a dependency graph $L$ be the maximum of the
number of edges coinciding at one vertex of $L$. The idea behind the usefulness of dependency graphs is that if the maximal degree
is not too large, one expects a Central Limit Theorem for the partial sums of the family  ${\{X_{\alpha}\}}_{\alpha\in \mathcal I}$.

\begin{example} \label{stexample}
A standard situation is that there is an underlying family of independent random variables $\{Y_i\}_{i \in \mathcal A}$, and each $X_{\alpha}$
is a function of the variables $\{Y_i\}_{i \in \mathcal A_{\alpha}}$, for some $\mathcal A_{\alpha} \subset \mathcal A$. With $\mathcal S = \{ \mathcal A_{\alpha}: 
\alpha \in \mathcal I \}$, the graph $L=L(\mathcal S)$ with vertex set $\mathcal I$ and edge set $\{ \alpha \beta: A_{\alpha} \cap A_{\beta} \not= \emptyset \}$
is a dependency graph for the family ${\{X_{\alpha}\}}_{\alpha\in \mathcal I}$. As a special case of this example, we will consider subgraphs
in the Erd\H{o}s-R\'enyi random graph model $G(n,p_n)$, that is $G$ has vertex set $\{1, \ldots, n\}$, and it has an edge between $i$ and $j$ with probability $p_n$, all these events being independent from each other. Let $\mathcal I$ be the set of 3-element subsets of $\{1, \ldots, n\}$, and if $\alpha = \{ i,j,k \} \in \mathcal I$, let $X_{\alpha}$ be the indicator function of the event {\it the graph $G$ contains the triangle with vertices $i,j$ and $k$}. Let $L$ be the graph with vertex set $\mathcal I$
and the following edge set: $\alpha$ and $\beta$ are linked if $|\alpha \cap \beta| =2$ (that is, if the corresponding triangles share an edge in $G$). Then $L$ is a dependency graph for the family $\{X_{\alpha} \}_{\alpha \in \mathcal I}$. 
\end{example}

Dependency graphs are used in geometric random graphs, see \cite{Penrose:2003}, and in geometric probability for statistics like the nearest-neighbour graph, the Delaunay triangulations
and the Voronoi diagramm of random point configurations, see \cite{PenroseYukich:2005}. More recently is has been used to prove asymptotic normality of pattern counts in random permutations in \cite{HitczenkoJanson:2010}.
Another context, outside the scope of the present paper, in which dependency graphs are used is the Lov\'asz Local Lemma,
see \cite{AlonSpencer08}. 

\noindent
We will consider the following setting: 

\begin{assumption}[Dependency-graph model] \label{dep-model}
From now on we consider the following model: Suppose that for each $n$, $\{X_{n,i}, 1 \leq i \leq N_n \}$ is a family of bounded random variables, $|X_{n,i}| \leq A_n$ a.s.
Suppose, in addition, that $L_n$ is a dependency graph for this family and let $D_n -1$ be the maximal degree of $L_n$. Let $Y_n := \sum_{i=1}^{N_n} X_{n,i}$ and
$\sigma_n^2 := \V(Y_n)$.
\end{assumption}

Precise normality criteria  for $(Y_n)_n$  using dependency graphs have been given in \cite{Janson:1988}, \cite{BaldiRinott:1989} and \cite{Mikhailov:1991}. In \cite{Janson:1988} the following normality criterion was proved: assume that there exists an integer $s$ such that $\bigl( \frac{N_n}{D_n} \bigr)^{\frac 1s} \frac{D_n}{\sigma_n} A_n \to 0 $ as $n \to \infty$. Then for the dependency graph model in \ref{dep-model},
$\frac{X_n - \E X_n}{\sigma_n}$ converges in distribution to a standard normally distributed random variable.

\begin{example}
We consider the $G(n,p_n)$-model in Example \ref{stexample} and take $Y_n$ to be the number of triangles.
Let $p_n$ be bounded away from 1. One has $N_n \asymp n^3$, $D_n \asymp n$ and $M_n$ =1. Since $\sigma_n^2 \asymp \max( n^3 p_n^3, n^4 p_n^5)$ (see \cite[Lemma 3.5]{JLR:2000}), the criterion is fulfilled if $p_n  \gg n^{- 1/3 + \varepsilon}$ for some $\varepsilon >0$. The asymptotic normality is in fact true under the less restrictive hypothesis $p_n \gg n^{-1}$, see \cite{Rucinski:1988}.
\end{example}

A uniform control on cumulants of $(Y_n)_n$ from Assumption \ref{dep-model} was first considered in \cite{Janson:1988}:
Under Assumption \ref{dep-model} one has that
\begin{equation} \label{eqcumulantRG}
\big| \Gamma_j(Y_n) \big|  \leq C_j N_n D_n^{j-1} A^j
\end{equation}
for some universal constant $C_j$ and any $j \geq 3$. Here it is assumed that $|X_{n,i}| \leq A$ for all $i$ and $n$, a.s.
In \cite{DE2013} is was proved that one can take $C_j = (2 e)^j (j!)^3$. The results was improved in \cite[Theorem 9.1.7]{modbook}: one can take
$C_j = 2^{j-1} j^{j-2}$ giving uniform bounds on cumulants.

\begin{definition} \label{defDNA}
A sequence $(Y_n)_n$ of real valued random variables admits a uniform control on cumulants with DNA $(D_n, N_n, A)$, if
$D_n = o(N_n)$, $N_n \to \infty$ as $n \to \infty$ and for all $j \geq 2$
\begin{equation} \label{DNA}
\big| \Gamma_j(Y_n) \big| \leq C_j N_n D_n^{j-1} A^j. 
\end{equation}
Here $A$ is a constant and $C_j$ is a constant only depending on $j$.
\end{definition}

\begin{remark}
The setting of Assumption \ref{dep-model} is an example for a uniform control on cumulants with DNA, see \eqref{eqcumulantRG}.
\end{remark}

\begin{theorem} \label{DNA-bound}
Assume that a sequence $(Y_n)_n$ of real valued random variables admits a uniform control on cumulants with DNA $(D_n, N_n, A)$. Assume moreover that
\begin{equation} \label{addcond}
\lim_{n \to \infty} \frac{\Gamma_2(Y_n)}{N_n D_n} = \sigma^2.
\end{equation}
Consider
$Z_n := \frac{Y_n}{\sigma_n}$ with $\sigma_n^2 := \V (Y_n)$. Then we obtain for even $k$ that
$$
\big| \E (Z_n^{k}) - \E (N^{k}) \big|  \leq C_1(k)  \frac{D_n}{N_n}.
$$

\end{theorem}

\begin{proof}
By assumption the cumulant bounds are of the form in Theorem \ref{thmcumulants} with $\gamma=0$, $C_{j, 0} = C_j A^j$
and with
$$
\Delta_n^{j-2} = \frac{\sigma_n^j}{N_n D_n^{j-1}}.
$$
Hence we have  $\Delta_n^2 = \biggl( \frac{\sigma_n^j}{N_n D_n^{j-1}} \biggr)^{\frac{2}{j-2}}$, which is depending on $j$. 
But with $\sigma_n^2 \asymp N_n D_n$ by  assumption \eqref{addcond} we have $\Delta_n^2 \asymp \frac{N_n}{D_n}$.
Now we can apply Theorem \ref{thmcumulants}.
\end{proof}

\begin{example}[Number of triangles in Erd\H{o}s-R\'enyi random graphs]
In the model of Example \ref{stexample} we take $p \in(0,1)$ being fixed. With $\sigma_n^2 \asymp \max( n^3 p_n^3, n^4 p_n^5)$,
we obtain $\sigma_n^2 \asymp n^4$. With  $N_n \asymp n^3$ and $D_n \asymp n$ we obtain that condition \eqref{addcond}
holds. Hence we can apply Theorem \ref{DNA-bound}: for even $k$ we have
$$
\big| \E (Z_n^{k}) - \E (N^{k}) \big|  \leq C_1(k)  \frac{1}{n^2}.
$$
\end{example}

\begin{example}[Number of subgraphs in Erd\H{o}s-R\'enyi random graphs]
Now we like to count the number of subgraphs isomorphic to a fixed graph $H$ with $k$ edges and $l$ vertices.
As a special case of Example \ref{stexample}, let $\{H_{\alpha}\}_{\alpha \in \mathcal I}$ be given subgraphs of the complete graph $K_n$ and let $I_{\alpha}$
be the indicator that $H_{\alpha}$ appears as a subgraph in $G(n,p_n)$, that is $I_{\alpha} = 1_{\{ H_{\alpha} \subset {\Bbb G}(n,p) \}}$, $\alpha \in \mathcal I$.
Then $L(S)$ with $S = \{e_{H_{\alpha}} : \alpha \in \mathcal I\}$ is a dependency graph with edge set 
$\{ \alpha \, \beta: e_{H_{\alpha}} \cap e_{H_{\beta}} \not= \emptyset \}$.
Here we take the family of subgraphs of $K_n$ that are isomorphic to a fixed graph $H$, denoted by $\{G_{\alpha}\}_{\alpha \in 
\mathcal I}$. 
Consider $X_{\alpha} = I_{\alpha} - {\Bbb E} I_{\alpha}$ and define the graph $L_n$ by connecting every pair of indices $\alpha$ and $\beta$ such that the corresponding
graphs $G_{\alpha}$ and $G_{\beta}$ have a common edge. This is evidently a dependency graph for $(X_{\alpha})_{\alpha \in A_n}$, see \cite[Example 6.19]{JLR:2000}.
The subgraph count statistic $Y$ is the sum of all $X_{\alpha}$. We prevent the dependence on $|\mathcal I|$ in our notion.
Again we only consider a fixed $p \in (0,1)$ to guarantee condition \eqref{addcond}: notice that for $p$ being fixed we have
\begin{equation} \label{variancebound}
\text{const.} \, n^{2l-2} p^{2k-1} (1-p)
\leq \V Y \leq \text{const.} \, n^{2l-2} p^{2k-1} (1-p) 
\end{equation}
by \cite[2nd section, page 5]{Rucinski:1988}. 
Moreover we have 
$$
D_n \leq k (n-2)_{l-2} -1 \leq k n^{l-2} -1
$$
(see \cite[page 369, last estimate]{DE2013}).
The number $N_n$ of the subgraphs in $K_n$ which are isomorphic to $H$ satisfies the inequality
$$
\left({n}\atop{l}\right)
\leq N_n \leq n_l=n (n-1)\cdots (n-l-1)\,.
$$
Hence $N_n \asymp n^{l}$ and condition \eqref{addcond} is fulfilled. Summarizing, the cumulants of $Y$ can be bounded as follows: for any $j \geq 3$
$$
\big| \Gamma_j(Y) \big| \leq j! C_j \, n^l (k n^{l-2})^{j-1}.
$$
With the lower bound \eqref{variancebound} we can bound the cumulants of $Z := \frac{X}{\sqrt{\V X}}$ for $j\geq 3$
as follows:
$$
\big| \Gamma_j(Z) \big| \leq \frac{j! C_j}{n^{j-2}}.
$$
Here the constant $C_j$ is also depending on $k$ and $l$. See also the proof of Theorem 2.3 in \cite{DE2013}.
Summarizing, applying Corollary \ref{classical} we obtain for fixed $p$ and for any subgraph $H$ with $k$ edges and $l$ vertices the bound
$$
\big| \E (Z^{m}) - \E (N^{m}) \big|  \leq C_1(m,l,k)  \frac{1}{n^2}
$$
for even $m$.
\end{example}

\section{Weighted dependency graphs}

Very recently, in \cite{Feray:2018} the concept of {\it weighted dependency graphs} was introduced. The concept includes the possibility of
having small weights $w_e \in [0,1]$ on the edges of the graph, which encode the dependency structure. Here  a weight 0 is the same as no edge.
The examples are sums of pairwise dependent random variables. For such families, the only usual dependency graph is the complete graph and the standard theory of dependency graphs is useless. 
Informally, that a family of random variables $\{ X_{n,i}, 1 \leq i \leq N_n \}$
admits a weighted graph $G$ as weighted dependency graph means that $G$ has vertex-set of size $N_n$, and the smaller the weight of an edge $\{a,b\}$
is, the closer to independent $X_{n,a}$ and $X_{n,b}$ should be. In particular, an edge of weight 0 means that $X_{n,a}$ and $X_{n,b}$ are independent.
Formally, this closeness to independence is not only measured by a bound on the covariance, but also involves bounds on higher order cumulants, see \cite[Definition 4.5]{Feray:2018}. 

To cut the story short, for each $n$, we consider a family $\{ X_{n,i}, 1 \leq i \leq N_n \}$ of
random variables with finite moments defined on the same probability space. We assume that for each $n$ one has a $(\Psi_n, C)$ weighted dependency graph
$L_n$ for $\{ X_{n,i}, 1 \leq i \leq N_n \}$ in the sense of Definition 4.5 in \cite{Feray:2018}, and we let $Y_n = \sum_{i=1}^{N_n} X_{n,i}$ and $\sigma_n^2 = \V(Y_n)$, and we assume that this sequence
admits a uniform control on cumulants with DNA $(Q_n, R_n ,1)$: We assume that $Q_n = o(R_n)$, $R_n \to \infty$ as $n \to \infty$ and for all $j \geq 1$,
\begin{equation} \label{cumulantboundwdg}
\big| \Gamma_j(Y_n) \big| \leq C_j R_n Q_n^{j-1},
\end{equation}
with a constant $C_j$ only depending on $j$. Although models with a corresponding weighted dependency graph are much more complicated concerning the
dependency structure, \cite{Feray:2018} has been successful in obtaining  examples, where the uniform control of the cumulants can be checked. 
As noticed in \cite[Section 4.3]{Feray:2018} in the special case $\Psi_n \equiv1$, the quantities $R_n$ and $Q_n$ in \eqref{cumulantboundwdg} can be replaced by $N_n$
(the number of vertices) and $D_n$ (the maximal weighted degree plus 1). In the following three examples, we restrict ourselves to this case:

\begin{example}[Crossings in random pair partitions]
A pair partition of $[2n] := \{1,2,\ldots, 2n\}$ is a set $H$ of disjoint 2-element subsets of $[2n]$ whose union is $[2n]$.
For each $i$ in $[2n]$ there is a unique $j \not= i$ such that $\{i,j\}$ is in $H$, the partner of $i$. A uniform random pair partition of $[2n]$ can be constructed as follows: Take $i_1$ arbitrarily and choose its partner $j_1$ uniformly at random among numbers different from $i_1$, i.e. each number different from $i_1$ is taken with probability $1/(2n-1)$. Then take $i_2$ arbitrarily, different from $i_1$ and $j_1$, and choose its partner $j_2$ uniformly at random among numbers different from $i_1, j_1$ and $i_2$ (with probability $1/(2n-3)$) and so on. 
 A {\it crossing} in a pair partition $H$ is a quadruple $(i,j,k,l)$ with $i < j< k <l$  such that  $\{i,k\}$ and $\{j,l\}$ belong to $H$.
 Now let $A_n$ be the set of two element subsets of $[2n]$. For $\{i,j\} \in A_n$ we define a random variable $X_{i,j}$ such that $X_{i,j}=1$, if $\{i,j\}$ belongs to the random pair partition $H_n$, and $0$ otherwise. Let $A_n'$ be the set of quadruples $(i,j,k,l)$ of elements of $[2n]$ with $i<j<k<l$. For $(i,j,k,l) \in A_n'$ we set
 $X_{i,j,k,l} := X_{i,k} X_{j,l}$. Hence this random variable has value $1$ if $(i,j,k,l)$ is a crossing in the random pair partition $H_n$, and $0$ otherwise.
 We consider the number of crossings in the random pair partition $H_n$
 $$
 Y_{n} := \sum_{i<j<k<l} X_{i,j,k,l}.
 $$
 In \cite[Theorem 6.5]{Feray:2018}, a CLT for $Z_n := (Y_n - \E Y_n)/ \sqrt{\V Y_n}$ was proved using the weighted dependency structure of this random variable.
 See \cite{cross} and references therein for numerous results on crossings.
 It was proven by showing that \eqref{cumulantboundwdg} holds true with a certain constant $C_j$, with $R_n \asymp n^2$ (see \cite[(6.3)]{Feray:2018}) and
 $Q_n=n$. Moreover, the variance of $Y_n$ was computed in \cite[Appendix B.1]{Feray:2018}, and we see that $\V Y_n \asymp n^3$. Hence assumption
 \eqref{addcond} holds and we obtain the bounds
$$
\big| \E (Z_n^{k}) - \E (N^{k}) \big|  \leq C_1(k)  \frac{1}{n}
$$
for even $k$.
\end{example}

\begin{example}[Subgraph counts in Erd\"os-R\'enyi model $G(n,m_n)$]
For each $n$, let $m_n$ be an integer between $0$ and ${n \choose 2}$. We now consider the Erd\"os-R\'enyi graph model $G(n,m_n)$, i.e. $G$ is a graph with vertex set $V=[n]$ and an edge set $E$ of size $m_n$, chosen uniformly at random among all possible edge sets of size $m_n$. We set $p_n := m_n / {n \choose 2}$. For any 2-element subset $\{i,j\}$ of $V$, we define $X_{i,j}$ such that $X_{i,j}=1$ if the edge $\{i,j \}$ belongs to the random graph $G$, and $0$ otherwise.
The value is $1$ with probability $p_n$. However, unlike in $G(n,p_n)$, these random variables are not independent. In \cite{Feray:2018}, a weighted dependency graph
in $(G(n,m_n)$ for the family $(X_{i,j})$ is presented.

Now fix a graph $H$ with at least one edge, and let $A_n^H$ be the set of subgraphs $H'$ of the complete graph $K_n$ on vertex set $[n]$ that are isomorphic to $H$. 
Let $G$ be a random graph with the distribution of the model $G(n,m_n)$. For $H'$ we write
$$
X_{H'} = \prod_{\{i,j\} \in E_{H'}} X_{i,j},
$$
and denote by
$$
Y_n^H = \sum_{H' \in A_n^{H'}} X_{H'}
$$
the number of subgraphs of $G$ that are isomorphic to $H$ (subgraph count statistic). In \cite[Proposition 7.2]{Feray:2018}, a weighted dependency graph for
the family $(X_{H'})_{H' \in A_n^H}$ was constructed. If $v_H$ denotes the number of vertices and $e_H$ the number of edges of $H$, we write
$$
\Phi_H := \min_{K \subset H, e_K >0} n^{v_k} p_n^{e_K}
$$
and
$$
\widetilde{\Phi}_H :=  \min_{K \subset H, eK >1} n^{v_k} p_n^{e_K}.
$$
In \cite[Theorem 7.5]{Feray:2018}, it was observed that \eqref{cumulantboundwdg} holds true with a certain constant $C_j$, with $R_n \asymp n^{v_H} p_n^{e_H}$ (see \cite[(7.3)]{Feray:2018}) and $Q_n= \frac{n^{v_H} p_n^{e_H}}{\Phi_H}$. Moreover we use the following estimate for the variance given in \cite[Lemma 7.3]{Feray:2018}:
\begin{equation} \label{variancegnm}
\V (Y_n^H) \geq C \frac{(n^{v_H} p_n^{e_H})^2}{\widetilde{\Phi}_H} (1-p_n)^2,
\end{equation}
for some constant $C>0$ and whenever $n (1-p_n)^2 \gg 1$ and $n$ is sufficiently large. Note that the variance of $Y_n^H$ has a different order of magnitude
than in the independent model $G(n,p_n)$, which was already observed in \cite{Janson:1990b}.

\begin{assumption} \label{ass2}
To be able to verify assumption \eqref{addcond}, we assume that $p \in (0,1)$ is fixed and $m_n \approx p {n \choose 2}$.
Moreover we assume that $H$ has a component with three vertices and two edges (a path $P_2$). 
\end{assumption}

The assumption implies that $\Phi_H \asymp \widetilde{\Phi}_H \asymp n^3$. Moreover we know that $\V (Y_n^H) \asymp n^{2 v_H-3}$ (whereas $\V (Y_n^H) \asymp n^{2 v_H -2}$ in the $G(n,p_n)$ random graph), see \cite[Example 6.55]{JLR:2000}.
We conclude that under Assumption \ref{ass2} we have
$$
\frac{\V(Y_n^H)}{R_n \, Q_n} \asymp {\text const.},
$$
and hence Assumption \eqref{addcond} is verified. Moreover we observe that
$$
\big| \Gamma_j(Y_n^H) \big| \leq C_j \, (n^{v_H} p^{e_H})^j \frac{1}{\Phi_H^{j-1}}.
$$
With the estimate \eqref{variancegnm}, we have with $Z_n^H = \frac{Y_n^H - \E(Y_n^H)}{\sqrt{\V(Y_n^H)}} $ that
$$
\big| \Gamma_j(Z_n^H) \big| \leq C_j(p) \frac{\widetilde{\Phi}_H^{j/2}}{\Phi_H^{j-1}} \leq \frac{\widetilde{C}_j(p)}{ \bigl( n^{3/2} \bigr)^{j-2}}. 
$$ 
With Theorem \ref{thmcumulants} or Theorem \ref{DNA-bound} we have proven: 

\begin{theorem}
Let $p \in (0,1)$ be fixed and $m_n \approx p {n \choose 2}$ and consider a random graph $G$ taken with Erd\"os-R\'enyi distribution $G(n, m_n)$.
Fix some graph $H$ that contains $P_2$. We denote by $Y_n^H$ the number of copies of $H$ in the random graph $G$. Then
with $Z_n^H = \frac{Y_n^H - \E(Y_n^H)}{\sqrt{\V(Y_n^H)}}$ we have for any even $k \geq 4$
$$
\big| \E \bigl( (Z_n^H)^{k} \bigr) - \E (N^{k}) \big|  \leq C_1(k)  \frac{1}{n^3}.
$$
\end{theorem}
\end{example}

\begin{example}[Spins in the $d$-dimensional Ising model]
The Ising model on a finite subset $\Lambda$ of $\Z^d$ is given by the Gibbs distribution
$$
\mu_{\Lambda, \beta, h}(\omega) = \frac{1}{Z_{\Lambda, \beta, h}} e^{- H_{\Lambda, \beta, h}}
$$
with
$$
H_{\Lambda, \beta, h} := - \beta \sum_{\{i,j\} \in {\mathcal E}_{\Lambda}} \sigma_i(\omega) \sigma_j(\omega) - h \sum_{i \in \Lambda} \sigma_i(\omega)
$$
for each $\omega = \bigl( \sigma_i(\omega) \bigr)_{i \in \Lambda}$ with $\sigma_i(\omega) \in \{-1,+1\}$. Here $h \in \R$ is called the magnetic field and $\beta>0$ the inverse temperature, and $ {\mathcal E}_{\Lambda}:= \{ \{i,j\} \subset \Lambda: \|i-j\|_1=1\}$ is the set of nearest neighbour pairs in $\Lambda$, measured in the
graph distance $\| \cdot \|_1$ in $\Z^d$. $Z_{\Lambda, \beta, h}$ is called the partition function.
All the quantities are with free boundary conditions so far, which means that the value of the spins outside of $\Lambda$ is not taken into consideration. Fixing a spin configuration $\eta \in \{-1, +1 \}^{\Z^d}$, we define a spin configuration in $\Lambda$ with {\it boundary condition $\eta$}
as an element of the set $\Omega_{\Lambda}^{\eta} := \{ \omega \in \{-1, +1\}^{\Z^d} : \omega_i = \eta_i \,\, \forall i \notin \Lambda \}$. Then the Hamiltonian is given by 
$$
H_{\Lambda, \beta, h}^{\eta} := - \beta \sum_{\{i,j\} \in {\mathcal E}_{\Lambda}^b} \sigma_i(\omega) \sigma_j(\omega) - h \sum_{i \in \Lambda} \sigma_i(\omega)
$$
with ${\mathcal E}_{\Lambda}^b := \{ \{i,j\} \subset \Lambda: \|i-j\|_1=1, \{i,j\} \subset \Lambda \not= \empty \}$. The corresponding probability distributions are denoted by
$\mu_{\Lambda, \beta, h}^{\eta}$. The most classical boundary conditions are the $+$ boundary condition, where $\eta_i = +1$ for all $i \in \Z^d$, and the $-$ boundary condition,  where $\eta_i = -1$ for all $i \in \Z^d$. Quantities with $+$ (resp. $-$) boundary condition are denoted with a superscript $+$ (or $-$ respectively), e.g. 
$\mu_{\Lambda, \beta, h}^{+}$.

We now take an increasing sequence $\Lambda_n$ of finite sets with $\bigcup_{n \geq 1} \Lambda_n = \Z^d$. It is well known that the sequence $
\bigl( \mu_{\Lambda_n, \beta, h}^{+} \bigr)_n$ converges in the weak sense to a measure denoted by $\mu_{\beta, h}^{+}$, as $n \to \infty$, see \cite[Chapter 3]{FV:book}.
In a high temperature regime with $\beta < \beta_1(d)$ and $h=0$ (meaning that there exists a $\beta_1(d)$) or in the presence of a magnetic field $h \not= 0$, the limiting measure is independent of the choice of the boundary conditions. At low temperature $\beta > \beta_2(d)$ and $h=0$, the limiting measure depends on the boundary conditions. Here, we restrict ourselves to $+$ boundary conditions to have a well defined limiting measure in all cases. We drop the superscript $+$ and denote the limiting measure by $\mu_{\beta, h}$. 

The decay of joint cumulants of the spins under $\mu_{\beta,h}$ has been studied in a few research articles. A good summary is
\cite[Theorem 1.1]{DousseFeray:2018} and reads as follows. For random variables $X_1, \ldots, X_j$ with finite moments, consider the joint cumulant as
$$
\Gamma(X_1, \ldots, X_j) = [t_1, \ldots, t_j] \log \E \exp(t_1 X_1 + \cdots + t_j X_j).
$$
Here $[t_1, \ldots, t_j] F$ stands for the coefficient of $t_1 \cdots t_j$ in the series expansion of $F$ in positive powers of $t_1, \ldots, t_j$. Note that $\Gamma_j(X) = \Gamma(X, \ldots, X)$.  

\begin{theorem} \label{Ising1}
For the Ising model on $\Z^d$ with parameters $(\beta,h)$, there exist positive constants $\varepsilon(d)<1, \beta_1(d), \beta_2(d)$ and $h(d)$ depending on the dimension $d$ 
with the following property. Assume that we are in one of the regimes $h > h(d)$, or  $h=0$ and $\beta < \beta_1(d)$, or  $h=0$ and $\beta > \beta_2(d)$. Then for
any $j \geq 1$, there exists a constant $C_j$ such that for all $A = \{ i_1, \ldots, i_j \} \subset \Z^d$, one has
$$
\Gamma_j^{\beta,h} (\sigma_{i_1}, \ldots, \sigma_{i_j}) \leq C_j \varepsilon(d)^{l_T(A)}.
$$
Here we consider the joint cumulants with respect to the measure $\mu_{\beta, h}$ and $l_T(A)$ denotes the minimum length of a tree connecting vertices of $A$.  
\end{theorem}

The bounds on joint cumulants had been translated in terms of weighted dependency graphs for the spin variables in \cite[Theorem 1.2]{DousseFeray:2018}:

\begin{theorem} \label{Ising2}
Let $\omega = \bigl( \sigma_i(\omega) \bigr)_{i \in \Z^d}$ be a spin configuration according to $\mu_{\beta,h}$, where either  
$h > h(d)$, or $h=0$ and $\beta < \beta_1(d)$, or $h=0$ and $\beta > \beta_2(d)$. Let $G$ be the complete weighted graph
with vertex set $\Z^d$, such that every edge $e=(i,j)$ has weight $w_e = \varepsilon(d)^{\frac{\| i - j\|_1}{2}}$, where $\varepsilon$ comes 
from Theorem \ref{Ising1}. Then $G$ is a $C$-weighted dependency graph (see \cite[Definition 2.3]{DousseFeray:2018})
for the family $\{ \sigma_i, i \in \Z^d \}$ and some $C=(C_r)_r$.
\end{theorem}

We now consider $\Lambda_n := [-n,n]^d$ the $d$-dimensional cube centred at $0$ of side length $2n$, and we consider the magnetization $S_n = \sum_{i \in \Lambda_n} \sigma_i$
and 
$$
Z_n := \frac{S_n - \E (S_n)}{\sqrt{ \V (S_n)}}.
$$
With \cite[Lemma V.7.1]{Ellis:LargeDeviations} we know that $\sigma^2 := \lim_{n \to \infty} \frac{\V(S_n)}{|\Lambda_n|}$ exists as an extended real number.
Moreover, it is known that $\sigma^2 >0$, and that it is finite in the three regimes of Theorem \ref{Ising1}, see 
\cite[Corollary 4.4 and the proof of Theorem 4.2]{DousseFeray:2018}. With Theorem \ref{Ising2}, the number of vertices of the weighted dependency graph on 
$\Lambda_n$ is $|\Lambda_n| = (2n+1)^d$.  The maximal weighted degree is
$$
D_n -1 = \max_{i \in \Lambda_n} \sum_{j \in \Lambda_n} \varepsilon^{\frac{\| i - j \|_1}{2}}.
$$
As presented in \cite{DousseFeray:2018}, this object is bounded by a constant. Hence we can apply Theorem \ref{DNA-bound} -- condition \eqref{addcond}
is satisfied. We have proved the result:

\begin{theorem} \label{Ising3}
Consider the Ising model on $\Z^d$, with inverse temperature $\beta$ and magnetic field $h$, such that either $h > h(d)$, or $h=0$ and $\beta < \beta_1(d)$, or
$h=0$ and $\beta > \beta_2(d)$. Then for even $k$ with $k \geq 4$, we have
$$
\big| \E_{\beta,h} \bigl( Z_n^{k} \bigr) - \E (N^{k}) \big|  \leq C_1(k)  \frac{1}{(2n+1)^d}.
$$
\end{theorem}

\begin{remark} As was pointed out in \cite{DousseFeray:2018}, local and global patterns of spins in the Ising model satisfy a central limit theorem as well.
For details see Theorem 1.3 and 1.4 in \cite{DousseFeray:2018}. For local patterns the result of Theorem \ref{Ising3} can be proved.
For global patterns of size $m$, at least in the case where the patterns consist of positive spins only, the same result follows from \cite[proof of Theorem 4]{DousseFeray:2018} with a constant $C_1(k,m)$, which is depending on the size $m$ as well. The details are omitted.
\end{remark}
\end{example}

\section{Non-degenerate $U$-statistics}
Let $X_1,\dots,X_n$ be independent and identically distributed random
variables with values in a measurable space $\mathcal X$. For a
measurable and symmetric function $h:{\mathcal X}^2\to \R$ we define
$$
U_n(h):= \frac{1}{\left(n\atop 2\right)}
\sum_{1\leq i_1<i_2 \leq n} h(X_{i_1},X_{i_2})\:,
$$
where symmetric means invariant under any permutation of its arguments.
$U_n(h)$ is called a {\it U-statistic} with {\it kernel} $h$ and
{\it degree} $2$.
Define the conditional expectation by
\begin{eqnarray*}
h_1(x_1)
&:=&\E\bigl[ h(x_1, X_2)\bigr]
\\
&=&\E\bigl[ h(X_1,X_m)\big| X_1=x_1\bigr]
\end{eqnarray*}
and the variance by $\sigma_1^2:=\V\bigl[h_1(X_1)\bigr]$.
A U-statistic is called {\it non-degenerate} if $\sigma_1^2>0$.  We consider U-statistics which are assumed to be non-degenerate. 
Assume that $0 < \sigma_1^2 < \infty$, and suppose that there exist constants $\gamma\geq1$ and $C>0$ such that
\begin{equation}\label{UStatKum}
\E\bigl[|h(X_1,X_2)|^j\bigr]\leq C^j (j!)^{\gamma}
\end{equation}
for all $j\geq 3$. According to \cite{Alesk:1990}, see \cite[Lemma 5.3]{SaulisStratulyavichus:1989},
the cumulants of $U_n$ can be bounded by
$$
|\Gamma_j(U_n)| < 2 e^{2(j-2)}\frac{2^j-1}{j} C^j (j!)^{1+\gamma} \frac{1}{n^{j-1}}
$$
for all $j=1,2,\dots,n-1$ and $n\geq 7$. The quite involved proof is presented in \cite{SaulisStratulyavichus:1989}.
The variance for the non-degenerate $U$-statistic is given by
$\V (U_n)= \frac{4 \sigma_1^2}{n} \frac{n-2}{n-1} + \frac{2 \sigma_2^2}{n(n-1)}$,
see Theorem 3 in \cite[chapter 1.3]{Lee:U-Statistics}. Hence there exists an $n_0\geq 7$
large enough such that $\sqrt{\V (U_n)}\geq \frac{e \sigma_1}{\sqrt{2n}}$.
The following bound holds for the cumulants of $Z_n := \frac{U_n}{\sqrt{\V(U_n)}}$:
$$
|\Gamma_j (Z_n) |\leq
(j!)^{1+\gamma} \left(\frac{2 \sqrt{2} e C(\sigma_1)}{\sqrt{n}}\right)^{j-2},
$$
for all $j= 3,\dots, n-1$ and $n\geq n_0$.
Applying Theorem \ref{thmcumulants}, we have for any even $k \geq 4$
$$
\big| \E (Z_n^{k}) - \E (N^{k}) \big|  \leq C_1(k)  \frac{1}{n}.
$$

\section{Characteristic polynomials in the circular ensembles}

Consider the characteristic polynomial $Z(\theta) := Z(U,\theta) =
\det\bigl(I-U e^{-i \theta}\bigr)$ of a unitary $n \times n$ matrix $U$. The matrix $U$ is considered as a random variable in the 
{\it circular unitary ensemble} (CUE), that is the unitary group $U(n)$ equipped with the unique translation-invariant (Haar) probability measure.
In \cite{KeatingSnaith:2000}, exact expressions for any matrix size $n$ are derived for the moments of $|Z|$, and from these the asymptotics of the value 
distribution and cumulants of the real and imaginary parts of $\log Z$ as $n \to \infty$ are obtained. In the limit, these distributions are independent and Gaussian.
In \cite{KeatingSnaith:2000} the results were generalized to the circular orthogonal (COE) and the circular symplectic (CSE) ensembles. 
Let us consider the representation of $Z(U, \theta)$ in terms
of the eigenvalues $e^{i \theta_k}$ of $U$:
$$
Z(U,\theta)= \det\bigl(I-U e^{-i \theta}\bigr) = \prod_{k=1}^n \bigl(1-e^{i(\theta_k-\theta)}\bigr).
$$
Now let $Z$ represent the characteristic polynomial of an $n \times n$ matrix $U$ in either the CUE ($\beta=2$), the COE ($\beta=1$), or the CSE ($\beta=4$).
The $C \beta E$ average can then be performed using the joint probability density for the eigenphases $\theta_k$
$$
\frac{(\beta/2)!^n}{(n \beta/2)! (2 \pi)^n} \prod_{1 \leq j < m \leq n} |e^{i \theta_j} - e^{i \theta_m}|^{\beta}.
$$
Hence the $s$-th moment of $|Z|$ is of the form
$$
\langle |Z|^s \rangle_{\beta} = \frac{(\beta/2)!^n}{(n \beta/2)! (2 \pi)^n} \int_0^{2 \pi} \cdots \int_0^{2 \pi} d\theta_1 \cdots d\theta_n 
\prod_{1 \leq j < m \leq n} |e^{i \theta_j} - e^{i \theta_m}|^{\beta} \times \bigg| \prod_{k=1}^n \bigl(1 - e^{i(\theta_k - \theta)} \bigr) \bigg|^s.
$$
This integral can be evaluated using Selberg's formula, see \cite{Mehta:book}, which leads to
$$
\langle |Z|^s \rangle_{\beta} = \prod_{j=0}^n \frac{\Gamma(1 + j \beta /2) \Gamma(1 + s + j \beta /2)}{(\Gamma( 1 + s/2 + j \beta /2))^2},
$$
where $\Gamma$ (without an index) denotes the Gamma function.
Hence $\log \langle |Z|^s \rangle_{\beta}$ has a simple form and, at the same time, by definition equals $\sum_{j \geq 1} \frac{\Gamma_j(\beta)}{j!} s^j$,
where $\Gamma_j(\beta)= \Gamma_j(\Re \log Z)$ denotes the $j$-th cumulant of the distribution of the real part of $\log Z$ under $C \beta E$. Differentiating $\log \langle |Z|^s \rangle_{\beta}$
one obtains
\begin{equation*} \label{vorbereitung}
\Gamma_j (\beta) = \frac{2^{j-1} -1}{2^{j-1}} \sum_{k=0}^{n-1} \psi^{(j-1)}(1 + k \beta /2),
\end{equation*}
where 
$$
\psi^{(j)}(z):= \frac{d^{j+1} \log \Gamma(z)}{dz^{j+1}}
= (-1)^{j+1} \int_0^\infty \frac{t^j e^{-zt}}{1-e^{-t}} dt
$$
for $z\in\mathbb C$ with $\Re z>0$ are the polygamma functions. In \cite[Section 4]{DE2013} we proved that
\begin{equation*}
\left| \Gamma_j\Bigl(\frac{\Re \log(Z)}{\sigma_{n,\beta}}\Bigr) \right|
\leq
(j!) \frac{1}{\sigma_{n,\beta}^{j-2}}
\left\{\begin{array}{ll}
2^j \frac{\pi^2}{3} & \text{for }\beta=1\\
4 \frac{\pi^2}{6} & \text{for }\beta=2\\
8 \frac{\pi^2}{6} & \text{for }\beta=4
\end{array}
\right.
\end{equation*}
for all $j\geq 3$, hence equation \eqref{eqcumulants} is satisfied for $\gamma=0$ and
$\Delta_n=  \sigma_{n,\beta}$.
The $j$-th cumulant of the distribution of the imaginary part of $\log Z$ can be bounded by
the $j$-th cumulant of the distribution of the real part of $\log Z$
for all $j\geq 3$, see \cite[eq. (62)]{KeatingSnaith:2000}.

For $\beta=2$ we know that $\sigma_{n,2}^2 \asymp \frac 12 \log n$, see \cite[eq. (45)]{KeatingSnaith:2000}.
Hence we have proved that for any even $k \geq 4$ and $Z_n = \frac{\Re \log(Z)}{\sigma_{n,2}}$ we have
$$
\big| \E (Z_n^{k}) - \E (N^{k}) \big|  \leq C_1(k)  \frac{1}{\log n}.
$$

\section{Determinants of random matrix ensembles and random simplices}
In this section we consider random determinants of certain random matrix ensembles.

\subsection{Laguerre ensemble} \label{LE}
Let us start with the following prototype of a random matrix ensemble from mathematical statistics. 
The study of sample covariance matrices is fundamental in multivariate statistics. Typically, one thinks of $p(n)$ variables $y_k$ with each variable
measured or observed $n$ times. One is interested in analysing the covariance matrix $A^t \, A$, with $A$ being the $n \times p(n)$ matrix with $p(n) \leq n$, and entries $y_k^{(j)}$ for $j=1, \ldots, n$ and $k=1,\ldots, p(n)$. If $A$ is chosen to be a Gaussian matrix
over $\re$, $\cn$ or $\qn$, the distribution of the $p(n) \times p(n)$ random matrix $A^{\dagger} A$ is called {\it Laguerre}
real, complex or symplectic ensemble. Here $A^{\dagger}$ denotes the transpose, the Hermitian conjugate or the dual of $A$
accordingly, when $A$ is real, complex or quaternion. The eigenvalues $(\lambda_1, \ldots, \lambda_{p(n)})$ are real and non-negative and it is a well known fact that the joint density function on the set $(0, \infty)^{p(n)}$ is 
$$
\frac{1}{Z_{n,p(n),\beta}} \prod_{1 \leq j < k \leq p(n)} |\lambda_j - \lambda_k|^{\beta} \prod_{k=1}^{p(n)} \bigl( \lambda_k^{\frac{\beta}{2} (n - p(n) +1) - 1} e^{-\frac{\lambda_k}{2}} \bigr)
$$
for $\beta =1,2,4$ respectively, see for example \cite[Proposition 3.2.2]{Forrester:book}. 
Using Selberg integration from \cite[(17.6.5)]{Mehta:book}, we obtain
$$
Z_{n,p(n),\beta} = 2^{\frac{\beta}{2} n p(n) - p(n)} \prod_{k=1}^{p(n)} \frac{\Gamma(1+\frac{\beta}{2} k) \Gamma( \frac{\beta}{2}(n -p(n)) + \frac{\beta}{2} k)}{\Gamma(1 + \frac{\beta}{2})}.
$$
Using this Selberg formula, one obtains directly that
\begin{eqnarray*}
\E \biggl[ \biggl( \det W_{n,p(n)}^{L, \beta} \biggr)^z \biggr] & = &2^{p(n) z} \prod_{k=1}^{p(n)} \frac{\Gamma\bigl(
\frac{\beta}{2} (n-p(n)+k) +z \bigr)}{\Gamma\bigl(\frac{\beta}{2} (n -p(n) +k)\bigr)} \\
&=& 2^{p(n) z} \prod_{k=1+n-p(n)}^{n} 
\frac{\Gamma\bigl(\frac{\beta}{2} k +z \bigr)}{\Gamma\bigl(\frac{\beta}{2} k\bigr)},
\end{eqnarray*}
where $W_{n,p(n)}^{L, \beta}$ denotes the $\beta$-Laguerre distributed random matrix of dimension  $p(n) \times p(n)$. This object is called the {\it Mellin transform}
of the determinant, which is defined for any $z \in \cn$ with ${\operatorname{Re}}(z) > -\frac{\beta}{2}$.

\noindent
We introduce the notion
\begin{equation} \label{L}
L(p,l,\alpha;z) = \log \biggl( \prod_{k=1}^p \frac{\Gamma(\alpha(k+l) + z)}{\Gamma( \alpha(k+l))} \biggr) ,
\end{equation}
with $p, l \geq 1$ and $z \in \cn$ with $\operatorname{Re}(z) > - \alpha$ and $\alpha \in \re$, and obtain
\begin{equation*} \label{MellinL1}
\log \E \biggl[ \exp \bigl( z \log \bigl( \det W_{n,p(n)}^{L, \beta} \bigr) \bigr) \biggr] = z p(n) \log 2 + L(p(n), n-p(n), \beta/2;z).
\end{equation*}
It follows that

 \begin{equation} \label{MellinL}
 \Gamma_j \big( \log \det W_{n,p(n)}^{L, \beta} \bigr) = \frac{d^j}{d z^j} 
 L(p(n), n-p(n), \beta/2;z)\bigg|_{z=0} + 1_{\{j=1\}} p(n) \log 2.
 \end{equation}
 
 In the case $p(n)=n$ of $n \times n$ matrices, asymptotic expansions of \eqref{MellinL} have been considered in \cite[Theorem 5.1]{Borgoetal:2017}. 
From a point of view of mathematical statistics, the number of variables $p(n)$ and the number
of measurements or observations $n$ are typically different. 
In \cite{EK:2017} asymptotic expansions have been developed for $n-p(n)$ equal to a constant $c>0$, or $n-p(n)$ is growing at a certain rate with $n$, as well as the case of a fixed number of variables $p$. 
A good overview of results for $\beta$-Laguerre ensembles is \cite{Bai:book} and \cite{Forrester:book}. 
In \cite{Jonsson:1982} one can find a very early result: the author proved a central limit theorem for $\det W_{n,n}^{L, 1}$, which is
$$
\frac{\log \det W_{n,n}^{L, 1} + n + \frac 12 \log n}{\sqrt{2 \log n}} \to N(0,1),
$$
where $N(0,1)$ denotes the standard Gaussian distribution. 

Our aim is to analyse the asymptotic behaviour of the first and second cumulant, and to bound higher order cumulants. With respect to random determinants
of random matrix ensembles, this goes back to \cite{DelannayLeCaer:2000}. For further details see \cite{DoeringEichelsbacher:2012b}.
In \cite{GKT:2018} the results of \cite{DoeringEichelsbacher:2012b} were applied to study volumes of random simplices.
 
 From now on we only consider the case $\beta=1$. For $\beta \not= 1$ the asymptotic behaviour (in $n$ and $p(n)$) of all cumulants of $\det W_{n,p(n)}^{L, \beta}$ only
 differs by some constants depending on $\beta$.
 
 The digamma function is defined as $\psi(z) = \psi^{(0)}(z) := \frac{d}{dz} \log \Gamma (z)$, and the polygamma functions
 $$
 \psi^{(j)}(z) := \frac{d^j}{dz^j} \psi(z) = \frac{d^{j+1}}{dz^{j+1}} \log \Gamma(z), \quad j \in \N.
 $$
 First we analyse the expectation of $\det W_{n,p(n)}^{L, \beta}$. For $j=1$, we have
 $$
 \frac{d}{dz} L(p(n),n-p(n), \frac{1}{2};z) \bigg|_{z=0} = \sum_{k=1}^{p(n)} \psi \bigl( \frac{1}{2} (k + n-p(n)) \bigr) = \sum_{k=1}^n \psi \bigl( \frac{k}{2}  \bigr) - \sum_{k=1}^{n-p(n)} \psi \bigl( \frac{k}{2}  \bigr).
 $$
 As $n \to \infty$, one has $\sum_{k=1}^n \psi \bigl( \frac{k}{2}  \bigr) \sim n \log n$, see for example \cite[relation (2.10) and (2.19)]{DoeringEichelsbacher:2012b}.
 Hence 
\begin{equation*}
\E \big( \log \det W_{n,p(n)}^{L, 1} \bigr) \sim \left\{ \begin{array}{ll}
n \log n  + p(n) \log 2& \text{for }n-p(n)=o(n)\\
p(n) \log (2n) & \text{for }  p(n) = o(n)\\ 
c\, n \log (2n) & \text{for } p(n) \sim c \,n \text{ for some } c \in (0,1).
\end{array}
\right.
\end{equation*}
Next we analyse the variance of $\log \det W_{n,p(n)}^{L, \beta}$. We obtain
 $$
 \frac{d^2}{dz^2} L(p(n),n-p(n), \frac{1}{2};z) \bigg|_{z=0} = \sum_{k=1}^{p(n)} \psi^{(1)} \bigl( \frac{1}{2} (k + n-p(n)) \bigr).
 $$
We collect some asymptotic relations and bounds for polygamma functions.

\begin{lemma}
Let $j \in \N$. Then as $|z| \to \infty$ in $| \text{arg } z | < \pi - \varepsilon$, one has
\begin{equation} \label{v1}
\psi^{(j)}(z) = (-1)^{j-1} \frac{(j-1)!}{z^j} + {\mathcal O} \bigl( \frac{1}{z^{j+1}} \bigr),
\end{equation}
and for all $z > 0$,
\begin{equation} \label{v2}
|\psi^{(j)}(z) | \leq \frac{(j-1)!}{z^j} + \frac{j!}{z^{j+1}}.
\end{equation}
Moreover we have
\begin{equation} \label{v3}
\sum_{k=1}^n \psi^{(1)} \bigl( \frac k2 \bigr) = 2 \log n + c +o(1)
\end{equation}
with an explicit constant $c=2(\gamma +1+ \frac{\pi^2}{8})$ with the Euler-Mascheroni constant $\gamma$.
\end{lemma}

\begin{proof}
The first asymptotic relation can be found in \cite{AbramowitzStegun:1964}, pp. 259-260. The representation of $\Gamma(z)^{-1}$ due to Weiserstrass is 
$\frac{1}{\Gamma(z)} = z e^{\gamma z} \prod_{k=1}^{\infty} \bigl( 1 + \frac zk \bigr) e^{-\frac zk}$. Differentiating $-\log \Gamma(z)$ leads to
$$
\psi(z) = - \gamma - \frac 1z + \sum_{k=1}^{\infty} \biggl( \frac 1k - \frac{1}{z+k} \biggr) = - \gamma + \sum_{n=0}^{\infty} \biggl( \frac{1}{n+1} - \frac{1}{z+n} \biggr).
$$
Therefore one obtains
\begin{equation} \label{decrease}
\psi^{(j)}(z) = (-1)^{j+1} j! \sum_{k=0}^{\infty} \frac{1}{(z+k)^{j+1}}.
\end{equation}
It follows that
$$
|\psi^{(j)}(z)| \leq \frac{j!}{z^{j+1}} + j! \int_z^{\infty} \frac{dx}{x^{j+1}} = \frac{j!}{z^{j+1}} + \frac{(j-1)!}{z^j},
$$
which is \eqref{v2}. The last asymptotic relation \eqref{v3} can be found in \cite[relations (2.14) and (2.21)]{DoeringEichelsbacher:2012b}.
\end{proof}

\noindent
With \eqref{v3} we obtain 
$$
\sum_{k=1}^{p(n)} \psi^{(1)} \bigl( \frac{1}{2} (k + n-p(n)) \bigr) = 2 \log n - 2 \log (n - p(n) +1) + {\mathcal O}(1) \sim 2 \log \frac{n}{n-p(n) +1}
$$
in the case $n-p(n) =o(n)$. If $p(n) =o(n)$, we apply \eqref{v1} to see that
$$
\sum_{k=1}^{p(n)} \psi^{(1)} \bigl( \frac{1}{2} (k + n-p(n)) \bigr) \sim 2  \frac{p(n)}{n}.
$$
Finally, with $p(n) \sim c \, n$, we apply \eqref{v3} to see
$$
\sum_{k=1}^{p(n)} \psi^{(1)} \bigl( \frac{1}{2} (k + n-p(n)) \bigr) = 2 \log n +c - 2 \log (n-p(n)) - c +o(1) = \log \frac{1}{1-c} +o(1).
$$
Hence 
\begin{equation} \label{variance}
\V \big( \log \det W_{n,p(n)}^{L, 1} \bigr) \sim \left\{ \begin{array}{ll}
2 \log \frac{n}{n-p(n)+1}  & \text{for }n-p(n)=o(n)\\
2 \frac{p(n)}{n} & \text{for }  p(n) = o(n)\\ 
2 \log \frac{1}{1-c} & \text{for } p(n) \sim c \,n \text{ for some } c \in (0,1).
\end{array}
\right.
\end{equation}

Finally we will bound the higher order cumulants. To this end we will combine results of \cite{DoeringEichelsbacher:2012b} and \cite{GKT:2018}.
By \eqref{decrease}, $| \psi^{(j-1)}(\cdot)|$ is decreasing, and therefore for $j \geq 3$:
$$
\big| \Gamma_j(\log \det W_{n,p(n)}^{L, 1}) \big| = \bigg| \sum_{k=1}^{p(n)} \psi^{(j-1)} \bigl( \frac{1}{2} (k + n-p(n)) \bigr) \bigg| \leq p(n)
 \bigg| \psi^{(j-1)} \bigl( \frac{1}{2} (1 + n-p(n)) \bigr) \bigg|.
$$
With \eqref{v2} we have $| \psi^{(j-1)}(z)| \leq 2 (j-1)! z^{1-m}$, $z \geq 1$. Hence
$$
\big| \Gamma_j(\log \det W_{n,p(n)}^{L, 1}) \big| \leq 2^j d^{j-1} p(n) (j-1)! n^{1-j},
$$
where $d$ is a constant such that $\frac{n-p(n)+1}{2} > \frac nd$, which is possible to choose in the cases $p(n)=o(n)$ and $p(n) \sim c\, n$.
The constant might depend on $c$, but is does not depend on $n$ or $p(n)$.
There is a very general bound for the higher order cumulants, which is valid for every choice of $p(n)$. For $j \geq 3$ we have
$$
\big| \Gamma_j(\log \det W_{n,p(n)}^{L, 1}) \big| = \bigg| \sum_{k=1}^{p(n)} \psi^{(j-1)} \bigl( \frac{1}{2} (k + n-p(n)) \bigr) \bigg| \leq \sum_{k=1}^n 
 \bigg| \psi^{(j-1)} \bigl( \frac{k}{2} \bigr) \bigg|. 
$$
With \eqref{v2} it follows that for any $j \geq 3$
$$
\big| \Gamma_j(\log \det W_{n,p(n)}^{L, 1}) \big| \leq 2^j \sum_{k \geq 1} \biggl( \frac{(j-1)!}{k^j} + \frac{(j-1)!}{4 k^{j-1}} \biggr) \leq 2^j \bigl( \zeta(3) + \frac 14 \zeta(2) \bigr)
 (j-1)! < 2^{j+1} (j-1)!,
$$
using $(j-2)! \leq \frac 12 (j-1)!$, and where $\zeta$ denotes the Riemann zeta function. Summarizing we obtain
\begin{equation} \label{cumL}
\big| \Gamma_j(\log \det W_{n,p(n)}^{L, 1}) \big| \leq \left\{ \begin{array}{ll}
2^j d^{j-1} p(n) (j-1)! n^{1-j} & \text{for } p(n) = o(n) \text{ or } p(n) \sim c \,n,\\
2^{j+1} (j-1)! & \text{for arbitrary }  p(n).
\end{array}
\right.
\end{equation}

Now we consider
$$
Z_{n,p(n)}^L := \frac{\log \det W_{n,p(n)}^{L, 1} - \E (\log \det W_{n,p(n)}^{L, 1})}{\sqrt{ \V(\log \det W_{n,p(n)}^{L, 1})}},
$$
and with \eqref{variance} and \eqref{cumL}, we get, for some constants $C_1(j)$ and $C_2(j)$, that
\begin{equation} \label{cumLZ}
\big| \Gamma_j(Z_{n,p(n)}^L) \big| \leq \left\{ \begin{array}{ll}
C_1(j) (j-1)! \frac{1}{\bigl( \sqrt{p(n) n}\bigr)^{j-2}}  & \text{for } p(n) = o(n) \text{ or } p(n) \sim c \,n,\\
C_2(j) (j-1)! \frac{1}{ \bigl( \sqrt{ \log \frac{n}{n-p(n) +1}} \bigr)^j} & \text{for }  n-p(n)=o(n).
\end{array}
\right.
\end{equation}
Now we can apply Corollary \ref{classical} to obtain:

\begin{theorem} \label{theoremRM}
For the $\log$-determinant of the Laguerre ensemble with $\beta=1$, we obtain the bounds
$$
\big| \E \bigl( (Z_{n,p(n)}^L)^{k} \bigl) - \E (N^{k}) \big|  \leq C_1(k)  \frac{1}{p(n) \, n}
$$
for $k$ being even and $p(n) =o(n)$ or $p(n) \sim c \, n$ for a fixed $c \in (0,1)$, and
$$
\big| \E \bigl( (Z_{n,p(n)}^L)^{k} \bigl) - \E (N^{k}) \big|  \leq C_1(k)  \frac{1}{ \log \frac{n}{n-p(n) +1}}
$$
for $k$ being even and $n-p(n) =o(n)$, including the case $n=p(n)$.
\end{theorem}

\subsection{Further random matrix ensembles}
In \cite{EK:2017} it was observed that many other random matrix models can be analysed knowing the behaviour of $L$ in \eqref{L}.

In Section 2.2 of \cite{EK:2017}, it was observed that for the Jacobi ensemble
\begin{equation*} 
\log \E \biggl[ \biggl( \det W_{p(n),n_1,n_2}^{J, \beta} \biggr)^z \biggr]  =L(p(n), n_1-p(n), \beta/2;z) - L(p(n), n_1+n_2-p(n), \beta/2;z),
\end{equation*}
where $W_{p(n),n_1,n_2}^{J, \beta}$ denotes the $\beta$-Jacobi distributed random matrix of dimension  $p(n) \times p(n)$.
Hence bounds on cumulants can be obtained starting with
$$
\Gamma_j \big( \log \det W_{p(n),n_1,n_2}^{J, \beta} \bigr) = \frac{d^j}{d z^j} 
\bigl( L(p(n), n_1-p(n), \beta/2;z) - L(p(n), n_1+n_2-p(n), \beta/2;z) \bigr) \bigg|_{z=0}.
$$

In \cite[Section 2.3]{EK:2017} for the Ginibre ensemble (starting with an arbitrary $n \times n$ matrix $A$ whose entries are independent real or complex Gaussian random variables with mean zero and variance one), it was observed that 
\begin{equation*} 
\log \E \biggl[ \biggl( \det W_{n}^{G, \beta} \biggr)^z \biggr] = \frac{nz}{2} \log \bigl( \frac{2}{\beta} \bigr)+ L(n,0,\beta/2;z).
\end{equation*}
Hence bounds on cumulants can be obtained starting with
$$
\Gamma_j \big( \log \det  W_{n}^{G, \beta}\bigr) = \frac{d^j}{d z^j} 
L(n, 0, \beta/2;z) \bigg|_{z=0} + 1_{\{j=1\}} \frac n2 \log \frac{2}{\beta}.
$$
In \cite[Section 2.4 and 2.5]{EK:2017} ten more random matrix models for mesoscopic normal-super\-con\-ducting structures were considered.
As we can see from \cite[(2.9) and (2.19)]{EK:2017}, all models can be analysed considering the $L$ in \eqref{L}.

\subsection{Random simplices}
If for $p(n) \leq n$, $X_1, \ldots, X_{p(n)+1}$ are independent random points in $\re^n$ which are distributed according to a multivariate Gaussian distribution
with density $f(|x|) = (2 \pi)^{-n/2} \exp( - \frac 12 |x|^2)$, $x \in \re^n$, we denote by $VP_{n,p(n)}$ the $p(n)$-dimensional volume of the {\it parallelotope} spanned by the points $X_1, \ldots, X_{p(n)}$. This is the determinant of the corresponding Gram matrix. It is known, see \cite{Mathai:1999}, that for all $m \geq 0$, the moments of order $2m$ of the volume fulfil
$$
\log \E \bigl( (VP_{n,p(n)})^{2m} \bigr) = m p(n) \log 2 + \log \prod_{k=1}^{p(n)} \frac{\Gamma\biggl(
\frac{1}{2} (n-p(n)+k) +m \biggr)}{\Gamma\biggl(\frac{1}{2} (n -p(n) +k)\biggr)}.
$$
The formula is a consequence of the so-called Blaschke-Petkantschin formula from integral geometry. Hence with \eqref{L}, we will study the asymptotics of
\begin{equation} \label{MellinPa1}
\log \E \bigl( (VP_{n,p(n)})^{z} \bigr) =  \frac z2 p(n) \log 2 + L \bigl( p(n), n-p(n), 1/2; z/2 \bigr),
\end{equation}
which is exactly the same as studying the asymptotic behaviour of the log-determinant of a Laguerre ensemble in the case $\beta=1$ for $z/2$ instead of $z$, see \eqref{MellinL}. Interestingly enough, the application of the Blaschke-Petkantschin formula is an alternative proof of the moment identity \eqref{MellinL}, which
in random matrix theory is proved with the help of Selberg integrals. We obtain
$$
\Gamma_j \bigl( \log \E \bigl( (VP_{n,p(n)})^{z} \bigr) \bigr) =  \frac{d^j}{d z^j} L \bigl( p(n), n-p(n), 1/2; z/2 \bigr) \bigg|_{z=0} + 1_{\{j=1\}} \frac{p(n)}{2} \log 2.
$$
The only difference to our results in Subsection \ref{LE} is that we have to use the identity
$$
\frac{d^j}{d z^j} L \bigl( p(n), n-p(n), 1/2; z/2 \bigr) \bigg|_{z=0} =  \frac{1}{2^j} \sum_{k=1}^{p(n)} \psi^{(j-1)} \bigl( \frac{1}{2} (k + n-p(n)) \bigr).
$$
Therefore we only have to deal with the pre-factor $\frac{1}{2^j}$, which only changes the constants $C_1(j)$ and $C_2(j)$ in Theorem \ref{theoremRM}.

\noindent
If we denote by $VS_{n,p(n)}$  the $p(n)$-dimensional volume of the {\it simplex} with vertices $X_1, \ldots, \allowbreak X_{p(n)+1}$, the moment formulas are very similar.
The following formulas were proved using the affine Blaschke-Petkantschin formula, see \cite{Miles:1971} and \cite{Chu:1993}. In the Gaussian model one obtains
\begin{equation*} \label{MellinSim1}
\log \E \bigl( (p(n)! \, VS_{n,p(n)})^{z} \bigr) =  \frac z2 \log (p(n)+1) + \log \E \bigl( (VP_{n,p(n)})^{z} \bigr),
\end{equation*}
where $\log \E \bigl( (VP_{n,p(n)})^{z} \bigr)$ is defined in \eqref{MellinPa1}. Again we can prove the same bounds as in Theorem \ref{theoremRM}.

Finally, in \cite{Mathai:2001}, the author studied the moments of order $2m$ of $VP_{n,p(n)}$ and of $V S_{n,p(n)}$, respectively, if the random points are distributed according to three other distributions, which are called the Beta model, the Beta prime model and the spherical model. All these models can be considered
in the same way.  Cumulant bounds can be found in \cite{GKT:2018}, given case by case. The order of the bounds are the same and hence
one can observe the same results as in Theorem \ref{theoremRM}.

\newcommand{\SortNoop}[1]{}\def\cprime{$'$} \def\cprime{$'$}
  \def\polhk#1{\setbox0=\hbox{#1}{\ooalign{\hidewidth
  \lower1.5ex\hbox{`}\hidewidth\crcr\unhbox0}}} \def\cprime{$'$}
\providecommand{\bysame}{\leavevmode\hbox to3em{\hrulefill}\thinspace}
\providecommand{\MR}{\relax\ifhmode\unskip\space\fi MR }
\providecommand{\MRhref}[2]{%
  \href{http://www.ams.org/mathscinet-getitem?mr=#1}{#2}
}
\providecommand{\href}[2]{#2}

\end{document}